\documentclass[12pt, a4paper, reqno]{amsart}
\usepackage{amsmath}
\usepackage{amsfonts}
\usepackage{amssymb}
\usepackage{amsthm}
\usepackage{url}

\newcommand{\R}{\mathbb{R}}

\def\vint_#1{\mathchoice%
          {\mathop{\kern 0.2em\vrule width 0.6em height 0.69678ex depth -0.58065ex
                  \kern -0.8em \intop}\nolimits_{\kern -0.4em#1}}%
          {\mathop{\kern 0.1em\vrule width 0.5em height 0.69678ex depth -0.60387ex
                  \kern -0.6em \intop}\nolimits_{#1}}%
          {\mathop{\kern 0.1em\vrule width 0.5em height 0.69678ex depth -0.60387ex
                  \kern -0.6em \intop}\nolimits_{#1}}%
          {\mathop{\kern 0.1em\vrule width 0.5em height 0.69678ex depth -0.60387ex
                  \kern -0.6em \intop}\nolimits_{#1}}}

\newcommand{\art}[6]{{\sc #1, \rm #2, \it #3 \bf #4 \rm (#5), \mbox{#6}.}}
\newcommand{\book}[3]{{\sc #1, \it #2, \rm #3.}}
\newcommand{\AND}{{\rm and }}
\newcommand{\p}{{$p\mspace{1mu}$}}

\newcommand{\loc}{_{\rm loc}}

\newcommand{\eps}{\varepsilon}

\theoremstyle{plain}
\newtheorem{theorem}[equation]{Theorem}
\newtheorem{lemma}[equation]{Lemma}

\newtheorem{proposition}[equation]{Proposition}

\numberwithin{equation}{section}

\theoremstyle{definition}

\theoremstyle{remark}
\newtheorem{remark}[equation]{Remark}
\title{On the problem of unique continuation for the
  \p-Laplace equation}

\author{Seppo Granlund} \address[S.G.]{University of Helsinki,
  Department of Mathematics and Statistics, P.O. Box 68, FI-00014
  University of Helsinki, Finland} \email{seppo.granlund@pp.inet.fi}

\author{Niko Marola} \address[N.M.]{University of Helsinki, Department
  of Mathematics and Statistics, P.O. Box 68, FI-00014 University of
  Helsinki, Finland} \email{niko.marola@helsinki.fi}

\date{}

\begin{document}

\keywords{Frequency function; \p-harmonic function.}

\subjclass[2010]{Primary: 35J92; Secondary: 35B60, 35J70.}

\begin{abstract}
  We study if two different solutions of the \p-Laplace equation
  $$\nabla\cdot(|\nabla u|^{p-2}\nabla u)=0,$$ where $1<p<\infty$, can
  coincide in an open subset of their common domain of definition. We
  obtain some partial results on this interesting problem.
\end{abstract}

\maketitle

\section{Introduction}

We consider the \p-Laplace equation in an open connected set
$G\subset\R^n$, $n\geq 2$,
\begin{equation} \label{eq:pLap} \Delta_pu := \nabla\cdot(|\nabla
  u|^{p-2}\nabla u)=0, \qquad 1<p<\infty.
\end{equation}
For $p=2$ we recover the Laplace equation $\Delta u=0$. We study the
question whether two different solutions to \eqref{eq:pLap} can
coincide in an open subset of their common domain of definition $G$.
%
%
%

This problem of unique continuation is still, to the best of our
knowledge, an open problem, except for the planar case $n=2$. The
planar case has been solved by Alessandrini~\cite{Aless87}, and by a
different approach by Manfredi~\cite{Manfredi} and Bojarski and
Iwaniec in \cite{BoIw}, as they have observed that the complex
gradient of a solution to \eqref{eq:pLap} is quasiregular. 

In addition, there are some recent partial results of the unique
continuation property for the game \p-Laplace equation on trees. We
refer to \cite{Rossi}.

We refrain from giving a detailed bibliographical account on the
literature on unique continuation results for linear elliptic
equations in divergence form. We refer to the papers
\cite{GarLinIndiana} and \cite{GarLinCPAM} by Garofalo and Lin, and to
a more recent paper by Alessandrini~\cite{Aless}, and suggest the
reader to consult also their bibliographies for more detailed
information on the subject.

In the present paper, we deal with the problem of unique continuation
by studying a certain generalization of Almgren's frequency function
for the \p-Laplacian. Our results, along with the notation and the
preliminary results, are stated in
\textsection~\ref{sect:Results}. The proofs can be found in
\textsection~\ref{sect:ProofweakDprop}--\ref{sect:Affine}.

\section{Results}
\label{sect:Results}

Let $G$ be an open connected subset of $\R^n$. We consider the
\p-Laplace equation \eqref{eq:pLap} in the weak form
\begin{equation} \label{eq:pLaplace}
\int_G|\nabla u|^{p-2}\nabla u\cdot\nabla\eta\, dx = 0,
\end{equation}
where $\eta\in C_0^\infty(G)$ and $1<p<\infty$. We refer the reader
to, e.g., Heinonen et al.~\cite{HKM} and Lindqvist~\cite{Lindqvist}
for a detailed study of the \p-Laplace equation and various properties
of its solutions. We mention in passing, however, that the weak
solutions of \eqref{eq:pLap} are $C\loc^{1,\alpha}(G)$, where $\alpha$
depends on $n$ and $p$. We refer to DiBenedetto~\cite{DiB},
Lewis~\cite{LewisIndiana}, and Tolksdorf~\cite{Tolksdorf} for this
regularity result. \emph{Hence, without loss of generality, we may redefine
$u$ so that $u\in W\loc^{1,p}(G)\cap C^1(G)$}.

Let us define the \emph{frequency function}
\begin{equation} \label{FF} F_p(r) = \frac{r\int_{B(z,r)}|\nabla u|^p\,
    dx}{\int_{\partial B(z,r)}|u|^p\, dS},
\end{equation}
where $\overline{B}(z,r)\subset G$; we denote
\[
D(r)=\int_{B(z,r)}|\nabla u|^p\, dx \quad \textrm{and} \quad
I(r)=\int_{\partial B(z,r)}|u|^p\, dS.
\]
Observe that $F_p(r)$ is not defined for such radii $r$ for which
$I(r)=0$. We remark that $F_p(r)$ is a generalization of the well
known Almgren frequency function
\begin{equation} \label{AlmFF} F_2(r) = \frac{r\int_{B(z,r)}|\nabla
    u|^2\, dx}{\int_{\partial B(z,r)}|u|^2\, dS}
\end{equation}
for harmonic functions in $\R^n$. To the best of our knowledge,
$F_p(r)$, $p\neq 2$, has not been previously studied in the
literature. It might be interesting to study other generalizations,
for instance, the case in which $r$ is replaced with $r^{p-1}$ in
\eqref{FF}. We have, however, omitted such considerations here.

The main results of the present paper are the following theorems.

\begin{theorem} \label{thm:weakDprop} Suppose $u\in C^1(G)$. Assume
  further that there exist two concentric balls $B_{r_b}\subset
  \overline{B}_{R_b}\subset G$ such that the frequency function
  $F_p(r)$ is defined, i.e., $I(r)>0$ for every $r\in(r_b,R_b]$, and
  moreover, $\|F_p\|_{L^\infty((r_b,R_b])}<\infty$. Then there exists
  some $r^\star\in(r_b,R_b]$ such that
\begin{equation} \label{eq:weakD1}
\int_{\partial B_{r_1}}|u|^p\, dS
  \leq 4\int_{\partial B_{r_2}}|u|^p\,
  dS,
\end{equation}
for every $r_1,\,r_2\in(r_b,r^\star]$. In particular, the following
weak doubling property is valid
\begin{equation} \label{eq:weakD2}
\int_{\partial B_{r^\star}}|u|^p\, dS
  \leq 4\int_{\partial B_r}|u|^p\,
  dS,
\end{equation}
for every $r\in(r_b,r^\star]$.
\end{theorem}

In the following we formulate a result which says that the local
boundedness of the frequency function implies certain vanishing
properties of the solution. In this respect the situation is similar
to the linear case $p=2$, and we thus generalize this phenomenon to
every $1<p<\infty$.

\begin{theorem} \label{thm:uniquecont} Suppose $u$ is a solution to
  the \p-Laplace equation in $G$. Consider arbitrary concentric balls
  $B_{r_b}\subset \overline{B}_{R_b}\subset G$. Assume the following:
  whenever $I(r)>0$ for every $r\in (r_b,R_b]$, then
  $\|F_p\|_{L^\infty((r_b,R_b])}<\infty$. Then if $u$ vanishes on some
  open ball in $G$, $u$ is identically zero in $G$.
\end{theorem}

It remains an open problem whether the frequency function $F_p(r)$ is
locally bounded for the solutions to the \p-Laplace equation. Local
boundedness combined with the method of the present paper would solve
the unique continuation problem for equation \eqref{eq:pLap}. 

In \textsection~\ref{sect:Affine} we study the question whether a
solution can coincide with an affine function without being
identically affine in the whole common domain of definition. Clearly
an affine function is a solution to the \p-Laplace equation. In
Proposition~\ref{prop:uniqueLinear}--\ref{prop:inforder} we
provide an answer to this question. This is a nonlinear generalization
of the corresponding phenomenon known for harmonic functions. Perhaps
surprisingly, this feature is rather easy to achieve while the
classical unique continuation principle for the \p-Laplace equation
still remains an open problem.

In \textsection~\ref{sect:Remarks} we discuss some observations which
might be of interest for further studies.

\subsection*{Preliminaries}

Throughout the paper $G$ is an open connected subset of $\R^n$, $n\geq
2$, and $1<p<\infty$. We use the notation $B_r = B(z,r)$ for
concentric open balls of radii $r$ centered at $z\in G$. Unless
otherwise stated, the letter $C$ denotes various positive and finite
constants whose exact values are unimportant and may vary from line to
line. Moreover, $dx=dx_1\ldots dx_n$ denotes the Lebesgue volume
element in $\R^n$, whereas $dS$ denotes the surface element. We denote
by $|E|$ the $n$-dimensional Lebesgue measure of a measurable set
$E\subseteq\R^n$. The characteristic function of $E$ is denoted by
$\chi_E$. Along $\partial B_r$ is defined the outward pointing unit
normal vector field at $x\in \partial B_r$ and is denoted by
$\nu(x)=(\nu_1,\ldots,\nu_n)(x)$. We will also write $u_\nu = \nabla
u\cdot \nu$ or $\partial u/\partial \nu$ for the normal
derivative of $u$. Define sets $P$ and $N$ as follows
\begin{equation*}
P = \{\omega\in \partial B_1: u(r\omega)>0\} \quad \textrm{and} \quad
N = \{\omega\in \partial B_1: u(r\omega)\leq 0\}
\end{equation*}
for all $r>0$.

We obtain the following formula for the derivative of $I(r)$ in
\eqref{FF}. Consult similar calculations in
Garofalo--Lin~\cite{GarLinIndiana} for the case $p=2$.
\begin{align} \label{I'} I'(r) & = \frac{\partial}{\partial
    r}\left(r^{n-1}\int_{\partial B_1}|u(r\omega)|^p\, d\omega\right) = \frac{n-1}{r}r^{n-1}\int_{\partial B_1}|u(r\omega)|^p\, d\omega \nonumber \\
  & \qquad + pr^{n-1}\left(\int_{\partial
      B_1}u(r\omega)^{p-1}\chi_Pu_r(r\omega)\, d\omega\right. \nonumber \\
  & \qquad \qquad \left. - \int_{\partial
      B_1}(-u(r\omega))^{p-1}\chi_Nu_r(r\omega)\, d\omega\right) \nonumber \\
  & \leq \frac{n-1}{r}r^{n-1}\int_{\partial B_1}|u(r\omega)|^p\, d\omega
  + pr^{n-1}\left(\int_{\partial B_1\cap
      P}|u(r\omega)|^{p-1}|u_r(r\omega)|\,d\omega\right. \nonumber \\
  & \qquad \qquad \left. + \int_{\partial B_1\cap
      N}|u(r\omega)|^{p-1}|u_r(r\omega)|\, d\omega\right).
\end{align}
Formula \eqref{I'} gives us the inequality
\begin{equation} \label{I'2}
I'(r) \leq \frac{n-1}{r}I(r) + p\int_{\partial B_r}|u|^{p-1}|u_\nu|\, dS.
\end{equation}

We shall also need the following formula for the solutions to the
\p-Laplace equation. It is probably earlier known in the literature,
but we provide it here due to the lack of references.

\begin{lemma} \label{lemma:energyformula} Suppose $u$ is a solution to
  the \p-Laplace equation in $G$. Then the following identity holds
  for the \p-Dirichlet integral
  \begin{equation} \label{eq:energyformula} \int_{B_r}|\nabla u|^p\,
    dx = \int_{\partial B_r}|\nabla u|^{p-2}uu_\nu\, dS
\end{equation}
for every $\overline{B}_r\subset G$.
\end{lemma}

\begin{proof}
  As in the classical case $p=2$ the proof is based on the
  Gauss--Green theorem. In the general case $1<p<\infty$, however, $u$
  is not necessarily in $C^2(G)$. Hence we have to use an
  approximation argument; we will use the approximation method
  presented by Lewis in \cite{LewisIndiana}. Consider a ball $B_r$ and
  a bounded open set $D$ such that $\overline{B}_r\subset
  \overline{D}\subset G$. Let $0<\eps<1$. Following
  \cite{LewisIndiana} we construct a sequence of functions $\hat
  u_\eps\in W^{1,p}(D)\cap C^\infty(D)$ such that they minimize the
  variational integral
\[
\mathcal{I}_\eps(\psi) = \int_D\left(|\nabla\psi|^2+\eps\right)^{p/2}\, dx,
\]
over all admissible functions in $\mathcal{F}_u(D) = \{v\in W^{1,p}(D):
v-u \in W_0^{1,p}(D)\}$. It is well known that the minimizing
function $\hat u_\eps$ is unique. The function $\hat u_\eps$ is a
solution to uniformly elliptic equation in the weak form
\[
\int_D\left(|\nabla\hat u_\eps|^2+\eps\right)^{p/2-1}\nabla\hat
u_\eps\cdot\nabla\eta\, dx = 0
\]
for all $\eta\in C_0^\infty(D)$, which is equivalent to
\[
\nabla\cdot\left(\left(|\nabla\hat u_\eps|^2 + \eps\right)^{p/2-1}\nabla\hat
u_\eps\right)= 0
\]
by the Gauss--Green theorem and the fact,
cf. Lewis~\cite{LewisIndiana}, that $\hat u_\eps\in C^\infty(D)$. Then
we consider the vector field
\[
U_{\eps} = \hat u_\eps(|\nabla\hat u_\eps|^2+\eps)^{(p-1)/2}\nabla\hat
u_\eps,
\]
It is clear that $U_\eps\in C^1(D)$. We may apply the Gauss--Green
theorem to $U_\eps$ and obtain the following formula
\begin{align} \label{eq:approx}
  \int_{B_r}|\nabla\hat u_\eps|^2 & \left(|\nabla\hat u_\eps|^2
    +\eps\right)^{p/2-1}\, dx \nonumber \\
  & \qquad = \int_{\partial B_r}\left(|\nabla\hat u_\eps|^2
    +\eps\right)^{p/2-1}\hat u_\eps\frac{\partial\hat
    u_\eps}{\partial\nu}\, dS.
\end{align}
Above we used the fact that
\[
\int_{B_r}\hat u_\eps\nabla\cdot\left(\left(|\nabla\hat u_\eps|^2
    +\eps\right)^{p/2-1}\nabla\hat u_\eps\right)\, dx = 0.
\]
Due to Lewis~\cite[Theorem 1]{LewisIndiana}, there exists $\alpha>0$,
depending only on $p$ and $n$, and positive $A<\infty$, depending only
on $p$, $n$, and $D$, such that
\begin{equation} \label{eq:gradest}
\max_{x\in \overline{D}}|\nabla\hat u_\eps(x)| \leq A,
\end{equation}
and for each $x,y\in D$
\begin{equation} \label{eq:Holderest}
|\nabla\hat u_\eps(x)-\nabla\hat u_\eps(y)| \leq A|x-y|^\alpha.
\end{equation}
In particular, constants $A$ and $\alpha$ are independent of
$\eps$. From \eqref{eq:gradest}, \eqref{eq:Holderest}, the Poincar\'e
inequality (see, e.g., \cite[Chapter 1]{HKM} or \cite[Chapter
7]{GilTru}) and from the weak compactness of $W^{1,p}$, it follows
that a subsequence of $\{\hat u_\eps\}$ converges weakly to a function
$v$ in $W^{1,p}$, and $v\in \mathcal{F}_u(D)$. To prove that $v$
minimizes the \p-Dirichlet integral $\mathcal{I}=\int_D|\nabla
\psi|^p\, dx$ over $\mathcal{F}_u(D)$, suppose
$\psi\in\mathcal{F}_u(D)$ is arbitrary. Since $\hat u_\eps$ is the
minimizing function we obtain
\[
\mathcal{I}(\psi) = \lim_{\eps\to 0}\mathcal{I}_\eps(\psi)\geq
\liminf_{\eps\to 0}\mathcal{I}_\eps(\hat u_\eps) \geq \mathcal{I}(v),
\]
where in the last inequality we used Reshetnyak's lower semicontinuity
theorem, see \cite[Theorem 1.1]{Reshetnyak}. Hence $v$ minimizes the
\p-Dirichlet integral in $\mathcal{F}_u(D)$, and so $v=u$.

To apply the Ascoli--Arzela principle we need to verify that the
sequences $\{\hat u_\eps\}$ and $\{\nabla\hat u_\eps\}$ are uniformly
bounded and equicontinuous. These two properties for the latter
sequence follow from \eqref{eq:gradest} and \eqref{eq:Holderest}. In
addition, equicontinuity of $\{\hat u_\eps\}$ follows from
\eqref{eq:gradest} and Morrey's lemma, see, e.g.,
\cite[\textsection~2.3, Lemma 4.1]{LaUr} or \cite[Chapter 12]{GilTru}. 
That the sequence is
uniformly bounded follows from the weak maximum principle of the
\p-Laplace equation.

The Ascoli--Arzela theorem implies that there exists a subsequence of
$\{\hat u_\eps\}$ and of $\{\nabla\hat u_\eps\}$, both still denoted
by $\{\hat u_\eps\}$ and $\{\nabla\hat u_\eps\}$, such that $\hat
u_\eps$ and $\nabla \hat u_\eps$ converge uniformly to $u$ and $\nabla
u$ in $\overline{D}$, respectively. We then obtain the identity
\eqref{eq:energyformula} by passing to the limit in \eqref{eq:approx}.
\end{proof}

\begin{remark} Equation \eqref{eq:energyformula} is a generalization
  of the corresponding equation for harmonic functions in $\R^n$
\[
\int_{B_r}|\nabla u|^2\, dx = \int_{\partial B_r}
uu_\nu\, dS.
\]
From this identity one deduces that the denominator of $F_2(r)$ in
\eqref{AlmFF} is non-decreasing, cf. the analogue of formula
\eqref{I'} in the case $p=2$ \cite[Equation 1.12]{GarLinIndiana}. {\it For general
  $1<p<\infty$ we do not know whether $I(r)$ is monotone.}

On a related note, one can even provide a characterization for
harmonic functions. If $u$ is harmonic, the preceding identity
is clearly valid. On the other hand, if the above identity holds the
Gauss--Green formula gives
\[
\int_{B_r}u\Delta u\, dx = 0
\] 
for every ball $B_r$. Hence, $u$ is harmonic in $\{x\in G:\ u(x)\neq
0\}$. The well known Rad\'o type theorem by Kr\'al~\cite{Kral} implies
that $u$ is harmonic in $G$.
\end{remark}

We can readily deduce the following from
Lemma~\ref{lemma:energyformula}.

\begin{lemma}
  Suppose $u$ is a solution to the \p-Laplace equation. Then the
  following inequality is valid
  \begin{equation} \label{eq:gradEst} p\int_{B_r}|\nabla u|^p\, dx
    \leq (p-1)\int_{\partial B_r}|\nabla u|^p\, dS + \int_{\partial
      B_r}|u|^p\, dS.
\end{equation}
\end{lemma}

\begin{proof}
  From \eqref{eq:energyformula} using Young's inequality we simply
  obtain
\begin{align*}
  \int_{B_r}|\nabla u|^p\, dx & \leq \int_{\partial B_r}|\nabla u|^{p-1}|u|\, dS \\
  & \leq \frac{p-1}{p}\int_{\partial B_r}|\nabla u|^p\, dS +
  \frac1{p}\int_{\partial B_r}|u|^p\, dS.
\end{align*}
\end{proof}

\begin{remark}
  Suppose $u$ is a solution to the \p-Laplace equation and let
  $Z=\{x\in B_r: u(x)=0\}$. If there exists a constant $0<\gamma<1$
  such that $|Z| \geq \gamma |B_r|$, then there exists a constant $C$,
  depending on $n$, $p$, and $\gamma$, such that,
  cf. Giusti~\cite[Theorem 3.17]{Giusti},
\begin{equation*} \label{eq:Poincare}
\int_{B_r}|u|^p\, dx \leq Cr^p\int_{B_r}|\nabla u|^p\, dx.
\end{equation*}
Plugging this into \eqref{eq:gradEst} we have
  \begin{equation*} \label{eq:PoincareEst} \int_{B_r}|u|^p\, dx \leq
    Cr^p\left(\int_{\partial B_r}|\nabla u|^p\, dS + \int_{\partial
        B_r}|u|^p\, dS\right),
\end{equation*}
where $C$ depends on $n$, $p$, and $\gamma$.
\end{remark}

\section{Proof of Theorem~\ref{thm:weakDprop}}
\label{sect:ProofweakDprop}

Let $u$ be an arbitrary function in $C^1(G)$ and consider two balls
$B_r\subset\overline{B}_s\subset G$ such that $0<I(r)\leq I(s)$ for
$r\leq s$, where $s$ is fixed, and $r,s\in (r_b,R_b]$. We stress that
we do not assume any monotonicity of the function $I$.

Integrate both sides of inequality \eqref{I'2} over $(r,s)$ to get the
following estimate
\begin{align} \label{eq:Iest} I(s)-I(r) & \leq
  (n-1)\int_r^s\frac{I(t)}{t}\, dt + p\int_r^s\biggl(\int_{\partial
    B_t}|u|^{p-1}|u_\nu|\, dS\biggr)\, dt \nonumber \\
  & \leq (n-1)I(s)\log\frac{s}{r} +
  \int_r^s\biggl((\eps_0 p)^{-1/(p-1)}\frac{p-1}{t^{1/(p-1)}}\int_{\partial B_t}|u|^p\, dS \nonumber \\
  & \qquad \qquad + \eps_0pt\int_{\partial B_t}|\nabla u|^p\,
  dS\biggr)\, dt \nonumber \\
  & \leq (n-1)I(s)\log\frac{s}{r} + (\eps_0
  p)^{-1/(p-1)}(p-1)I(s)\int_r^st^{-1/(p-1)}\, dt \nonumber \\ &
  \qquad \qquad + \eps_0ps\int_{B_s}|\nabla u|^p\, dx.
\end{align}
We applied above Young's inequality $$ab\leq \eps_0 a^p + (\eps_0
p)^{-q/p}q^{-1}b^{q},$$ $\eps_0>0$, in the case in which
$a=|u_\nu|t^{1/p}$ and $b=|u|^{p-1}t^{-1/p}$. We shall fix $\eps_0$
later. We divide inequality \eqref{eq:Iest} by $I(s)$ and obtain
\begin{align} \label{FFestimate}
  \frac{I(s)-I(r)}{I(s)} & \leq (n-1)\log\frac{s}{r} + (\eps_0
  p)^{-1/(p-1)}(p-1)\int_r^st^{-1/(p-1)}\, dt \nonumber \\
& \qquad \quad + \eps_0 pF_p(s)
\end{align}
for every $r\leq s$, where $s$ is fixed, and $r,s \in (r_b,R_b]$ such
that $I(r)\leq I(s)$. Since the frequency function $F_p(r)$ is locally
bounded by the hypothesis of the theorem, we denote
$M=\|F_p\|_{L^\infty((r_b,R_b])} <\infty$. In addition, we note that
functions $\log\frac{s}{r}$ and $\int_r^st^{-1/(p-1)}\, dt$ in
\eqref{FFestimate} tend to zero when $r$ goes to $s$. In order to get
each term on the right-hand side in \eqref{FFestimate} smaller than,
say, $1/4$ we first set $\eps_0 = 1/(4pM)$ and then choose a radius
$r_0\in (r_b,R_b]$ so close to $r_b$ such that
\[
(n-1)\log\frac{s}{r} \leq \frac1{4} \quad \textrm{and} \quad  (\eps_0
  p)^{-1/(p-1)}(p-1)\int_{r}^s t^{-1/(p-1)}\, dt \leq \frac1{4}
\]
for each $r\leq s$, where $s$ is fixed, and $r,s\in (r_b, r_0]$. Since
$r\mapsto I(r)$ is continuous on $[r_b,r_0]$, there exists a radius
$r^\star \in (r_b, r_0]$ such that
\[
I(r^\star) = \max_{r\in [r_b, r_0]}I(r).
\]
Then we clearly have $0< I(r) \leq I(r^\star)$ for each $r\in
(r_b,r_0]$. Therefore, by the above reasoning, we obtain for $r\in
(r_b,r^\star]\subset (r_b,r_0]$ the following inequality
\[
\frac{I(r^\star)-I(r)}{I(r^\star)} \leq \frac{3}{4},
\]
and hence a weak doubling property for all radii $r\in (r_b,r^\star]$
\[
I(r^\star) \leq 4I(r).
\]
We stress here that although the constant in the preceding weak
doubling property is uniform, the radius $r^\star$ depends on the
function $u$. Inequality \eqref{eq:weakD1} follows from the fact that
$r^\star$ provides the maximum value of $I(r)$ on $[r_b,r_0]$.

\section{Proof of Theorem~\ref{thm:uniquecont}}
\label{sect:ProofUniqCont}

Suppose on the contrary that the function $u$, a non-trivial solution
to the \p-Laplace equation \eqref{eq:pLap}, vanishes in a ball
$\overline{B}_{r_1}$ but $u$ is not identically zero in a concentric
open ball $B_{r_2}$, where $\overline{B}_{r_2}\subset G$. We remark
that the frequency function, $F_p(r)$, is not defined on $[0,r_1]$.

Let $t>0$ and consider an open ball $B_t$ which is concentric with
$B_{r_1}$ and $B_{r_2}$. Define
\[
s=\sup\{t>0: u|_{\partial B_t} \equiv 0\}.
\]
The aforementioned assumptions imply that $s\in [r_1,r_2)$. We note,
in addition, that due to Lemma~\ref{lemma:energyformula} we may
conclude that $u|_{\partial B_\rho}$ does not vanish identically for any
radii $\rho\in (s,r_2]$, hence $I(\rho)\neq 0$; here we could also
have applied the weak maximum and minimum principle instead of
Lemma~\ref{lemma:energyformula}.

The frequency function $F_p(r)$ is defined on $(s,r_2]$, moreover
$r\mapsto F_p(r)$ is absolutely continuous on this half open interval,
and by the hypothesis of the theorem $F_p(r)$ is bounded on
$(s,r_2]$. Theorem~\ref{thm:weakDprop} implies the existence of a
radius $r^\star\in (s,r_2]$ such that the following weak doubling
property holds
\[
I(r^\star) \leq 4I(r),
\]
for every $r\in(s, r^\star]$. Since $I(r)\to 0$ as $r\to s$ we have
reached a contradiction.

\section{Solutions coinciding with affine functions}
\label{sect:Affine}

Let us first state the following coincidence property of the solutions
to the \p-Laplace equation \footnote{We would like to thank an
  anonymous reader for a feedback which led to
  the current formulation of Proposition~\ref{prop:uniqueLinear}.}, which must be well known to the experts. Clearly an affine
function satisfies the \p-Laplace equation.

\begin{proposition} \label{prop:uniqueLinear} Suppose $u$ is a
  solution to the \p-Laplace equation in $G$. Consider an affine
  function $L(x) = l(x) + l_0$, where $l_0\in\R$ and $$l(x) =
  \sum_{i=1}^n\alpha_ix_i \neq 0.$$ If $u(x)=L(x)$ in an open subset
  $D\subset G$, then $u(x)=L(x)$ for every $x\in G$.  
  %
\end{proposition}

Proposition~\ref{prop:uniqueLinear} could also be stated as follows:
Suppose $u$ and $v$ are two solutions to the \p-Laplace equation in
$G$. Assume further that $\nabla v\neq 0$ in $G$. If $u(x)=v(x)$ in an
open set $D\subset G$, then $u(x)=v(x)$ for every $x\in G$.

\begin{proof}[Proof of Proposition~\ref{prop:uniqueLinear}]
We assume that there is an open set in $G$ where $u$ coincides with
$L$; Obviously, we can then consider the maximal open set $D$ where
$u$ coincides with $L$. Hence, we may
conclude that $\nabla u = \nabla L \not = 0$ on
$\overline{D}$. Therefore, by continuity of $\nabla u$ there exists
a neighborhood $B_\delta$, $\delta >0$, of a point $x_0\in \partial D$
such that $|\nabla u |\geq C>0$ on $B_\delta(x_0)$. It is known, and
we refer to Lewis~\cite[p. 208]{Lewis}, that $u$ is real analytic in
the open set where $\nabla u \not =0$. Thus $u(x) = L(x)$ on
$B_\delta$. The coincidence set can be expanded so that $D=G$.
%
\end{proof}

\subsection*{Discussion under extra regularity assumption and without
  using real analyticity}

We recall first that a function $f\in L\loc^q(G)$ vanishes of infinite
order at $x_0\in G$ if for some $q>0$
\[
\limsup_{r\to 0}r^{-k}\int_{B(x_0,r)}|f|^q\, dx = 0
\]
is valid for each $k\in\R$, $k>0$.

\begin{proposition} \label{prop:inforder} Suppose $u\in C^2(G)$ is a
  solution to the \p-Laplace equation in $G$ and that $L$ is an affine
  function, not identically zero. If $u-L$ vanishes of infinite order
  at $x_0\in G$, then $u$ coincides with $L$ in $G$.
\end{proposition}

\begin{remark}
  We stress that the result in Proposition~\ref{prop:inforder} is true
  without assuming that $u\in C^2(G)$. The proof is analogous to that
  of Proposition~\ref{prop:uniqueLinear} and uses real
  analyticity. However, in the $C^2$-case we do not appeal to real
  analyticity; instead we show that $u-L$ satisfies certain \emph{linear
  Laplace-type equation with a drift term} and then we are able to apply the
  frequency function approach by Garofalo and Lin. In addition, the
  obtained linear equation is of independent interest and useful for
  further studies.
\end{remark}

\begin{proof}[Proof of Proposition~\ref{prop:inforder}]
Note that the \p-Laplace equation, \eqref{eq:pLap}, can be written in
a different form as follows
\[
\Delta_p u = |\nabla u|^{p-4}\biggl(|\nabla u|^2\Delta u +
(p-2)\sum_{i,\,j =1}^nu_{x_i}u_{x_j}u_{x_i x_j}\biggr) = 0,
\]
and we may study the equation
\begin{equation} \label{eq:pLapmod}
|\nabla u|^2\Delta u + (p-2)\sum_{i,\,j =1}^nu_{x_i}u_{x_j}u_{x_i x_j}=0.
\end{equation}
Equation \eqref{eq:pLapmod} characterizes the weak solutions $u \in
C^2(G)$ of the \p-Laplace equation. We invoke Juutinen et
al.~\cite{JLM} and Lindqvist~\cite{Lindqvist} for this nontrivial
fact. Consider affine function $L(x)=l(x) + l_0$, $l_0\in\R$,
$l(x)\neq 0$. We shall show that the difference $u-L$, where $u$ is a
solution to \eqref{eq:pLapmod}, satisfies a modified uniformly
elliptic equation of the form
\begin{equation*} \label{eq:elliptic}
\sum_{i,\,j =1}^na_{ij}v_{x_i x_j}+\sum_{i=1}^nb_{i}(x)v_{x_i}=0
\end{equation*}
with constant coefficients $(a_{ij})$ and the drift term $b_i(x)$ is
continuous in $G$. Clearly, $\Delta_p L = 0$. Let
$\alpha:=(\alpha_1,\ldots,\alpha_n) = \nabla L$, and we denote the
difference $u-L$ by $h$. We proceed by manipulating \eqref{eq:pLapmod}
as follows
\begin{align*}
  0 & = |\nabla u|^2\Delta u + (p-2)\sum_{i,\,j =1}^nu_{x_i}u_{x_j}u_{x_i x_j} \\
  & = |(\nabla u-\alpha)+\alpha|^2\Delta u \\
  & \qquad + (p-2)\sum_{i,\,j
    =1}^n\left((u_{x_i}-\alpha_i)+\alpha_i\right)\left((u_{x_j}-\alpha_j)+\alpha_j\right)u_{x_i x_j} \\
  & = \left(|\nabla u-\alpha|^2+2(\nabla u-\alpha)\cdot\alpha+|\alpha|^2\right)\Delta u \\
  & \qquad \quad + (p-2)\left(\sum_{i,\,j
      =1}^n(u_{x_i}-\alpha_i)(u_{x_j}-\alpha_j)u_{x_i x_j} \right. \\
  & \qquad \qquad \left. +\sum_{i,\,j
      =1}^n\alpha_j(u_{x_i}-\alpha_i)u_{x_i x_j} +\sum_{i,\,j
      =1}^n\alpha_i(u_{x_j}-\alpha_j)u_{x_i x_j} \right. \\
  & \qquad \qquad \quad \left. +\sum_{i,\,j
      =1}^n\alpha_i\alpha_ju_{x_i x_j}\right).
\end{align*}
After rearranging the terms we obtain
\begin{align*}
  |\alpha|^2\Delta u & + (p-2)\sum_{i,\,j =1}^n\alpha_i\alpha_ju_{x_i
    x_j} + |\nabla u-\alpha|^2\Delta u + 2(\nabla u-\alpha)\cdot\alpha\Delta u \\
  & \qquad + (p-2)\left(\sum_{i,\,j
      =1}^n(u_{x_i}-\alpha_i)(u_{x_j}-\alpha_j)u_{x_i x_j} \right. \\
  & \qquad \quad + \left.\sum_{i,\,j
      =1}^n\alpha_j(u_{x_i}-\alpha_i)u_{x_i x_j} +\sum_{i,\,j
      =1}^n\alpha_i(u_{x_j}-\alpha_j)u_{x_i x_j}\right) = 0.
\end{align*}
Clearly $\Delta u = \Delta h$ and $\nabla h = \nabla u-\alpha$, thus
we get the following equation
\begin{align*}
  & |\alpha|^2\Delta h + (p-2)\sum_{i,\,j =1}^n\alpha_i\alpha_jh_{x_i
    x_j} + |\nabla h|^2\Delta h + 2(\nabla h\cdot\alpha)\Delta h \\
  & + (p-2)\left(\sum_{i,\,j
    =1}^nh_{x_i}h_{x_j}h_{x_i x_j} + \sum_{i,\,j =1}^n\alpha_jh_{x_i}h_{x_i
    x_j}+\sum_{i,\,j =1}^n\alpha_ih_{x_j}h_{x_i x_j}\right) = 0.
\end{align*}
By inspecting this last equation we observe that it can be
written in the following form
\begin{equation} \label{eq:almostlinear}
  |\alpha|^2\Delta h + (p-2)\sum_{i,\,j =1}^n\alpha_i\alpha_jh_{x_i
    x_j} + \sum_{i=1}^nb_i(x)h_{x_i}=0,
\end{equation}
where we have written
\[
b_i(x) = \Delta h\left(h_{x_i}+2\alpha_i\right) + (p-2)\sum_{j=1}^nh_{x_ix_j}\left(h_{x_j}+2\alpha_j\right).
\]
We study the quadratic form in \eqref{eq:almostlinear}. By
the Schwarz inequality,
\[
\sum_{i,\,j =1}^n\alpha_i\alpha_j\xi_i\xi_j =
\left(\sum_{i=1}^n\alpha_i\xi_i\right)^2 \leq |\alpha|^2|\xi|^2.
\]
Thus, for $p\geq 2$ we obtain 
\[
(p-1)|\alpha|^2|\xi|^2 \geq |\alpha|^2|\xi|^2 + (p-2)\sum_{i,\,j
  =1}^n\alpha_i\alpha_j\xi_i\xi_j \geq |\alpha|^2|\xi|^2 > 0,
\]
whereas if $1<p<2$ we deduce 
\[
|\alpha|^2|\xi|^2 \geq |\alpha|^2|\xi|^2 + (p-2)\sum_{i,\,j
  =1}^n\alpha_i\alpha_j\xi_i\xi_j \geq (p-1)|\alpha|^2|\xi|^2>0.
\]
Hence, the quadratic form is positive definite, and equation
\eqref{eq:almostlinear} is uniformly elliptic for all
$1<p<\infty$. Moreover, since the principal part coefficients are
constants \eqref{eq:almostlinear} can be written in the divergence
form.

Due to results by Garofalo and Lin in \cite{GarLinCPAM} and
\cite{GarLinIndiana}, the strong unique continuation principle is
valid for the equation
\begin{equation} \label{eq:GarLin}
-\nabla\cdot(A(x)\nabla u) + b(x)\cdot\nabla u + V(x)u=0,
\end{equation}
where $A(x) = (a_{ij}(x))_{i,j=1}^n$ is a real symmetric matrix-valued
function satisfying the uniform ellipticity condition and it is
Lipschitz continuous. The lower order terms, the drift coefficient
$b(x)$ and the potential $V(x)$, are even allowed to have
singularities.  The reader should consult (1.4)--(1.6) in
\cite{GarLinCPAM} for the exact structure conditions of $b$ and
$V$.

To be more precise, one of the main results in \cite{GarLinCPAM} is
that if $v$ is a solution to \eqref{eq:GarLin} in $G$, then $v$
satisfies the following doubling property
\begin{equation} \label{eq:doublingprop}
\int_{B_{2r}}v^2\, dx \leq C\int_{B_r}v^2\, dx,
\end{equation}
where $\overline{B}_{2r}\subset B_{\bar r}\subset G$, and the constant
$C$ depends on $n$, $v$, the ellipticity and the Lipschitz constant of
$A(x)$, and the local properties of $b(x)$ and $V$, and $\bar r$
depends on the aforementioned parameters but not on the function
$v$. See \cite{GarLinCPAM} for more details. Then if $v$ vanishes of
infinite order at $x_0\in G$,
$v$ must vanish identically in $G$. This is a consequence of
\eqref{eq:doublingprop}, we consult the proof of Theorem 1.2 in
\cite{GarLinIndiana} for this fact. Also the following follows by such
reasoning: if $v$ vanishes identically on a subdomain of $G$, then it
vanishes on the whole $G$, see Tao--Zhang~\cite[Corollary
2.6]{TaoZhang}.

To conclude, since our equation \eqref{eq:almostlinear} is of the type
\eqref{eq:GarLin} with $V\equiv 0$ and the drift term, $b(x)$, is
continuous, we obtain the claim from the results in \cite{GarLinCPAM}
as explained above.
\end{proof}

\begin{remark} \label{rmk:General} An argument many ways analogous to
  the preceding proof justifies the following more general claim:
  Suppose $u,\,v\in C^2(G)$ are two solutions to the \p-Laplace
  equation in $G$. Assume further that $\nabla v\neq 0$ in $G$. Then
  if $u-v$ vanishes of infinite order at $x_0\in G$, it must be that
  $u$ coincides with $v$ in the whole $G$. 

  This observation is obtained by considering \eqref{eq:pLapmod} which
  is satisfied by both $u$ and $v$. By subtracting and denoting
  $h=u-v$ we end up having the following equation in nondivergence
  form
\begin{align} \label{eq:Heq}
  |\nabla v|^2\Delta h & + (p-2)\sum_{i,j=1}^n v_{x_i}v_{x_j}h_{x_ix_j}
  +
  \left((\nabla v+\nabla u)\cdot\nabla h\right)\Delta u \nonumber \\
  & + (p-2)\sum_{i,j=1}^n u_{x_ix_j}\biggl(v_{x_i}h_{x_j} +
  u_{x_j}h_{x_i}\biggr)=0.
\end{align}
Since $u$ and $v$ are in $C^2(G)$ it is well known that equation
\eqref{eq:Heq} can be rewritten in the divergence form, see, e.g.,
\cite[\textsection 6]{Evans}. In addition, equation \eqref{eq:Heq} in
the divergence form is uniformly elliptic for all $1<p<\infty$ since
$\nabla v\neq 0$ in $G$. A reasoning similar to the one in the
preceding proof gives the claim.
\end{remark}

\section{Further remarks}
\label{sect:Remarks}

We close the paper by giving a few remarks which might be of interest
for further studies.

Suppose $u$ is a non-trivial solution to the \p-Laplace
equation. Assume further that there exists a positive constant
$A<\infty$ such that for any $\overline{B}_r\subset G$
\begin{equation} \label{eq:D'<I}
\int_{\partial B_r}|\nabla u|^p\, dS \leq A\int_{\partial B_r}|u|^p\, dS.
\end{equation}
Combining \eqref{eq:D'<I} with \eqref{eq:gradEst} we obtain
\[
\int_{B_r}|\nabla u|^p\, dx \leq C\int_{\partial B_r}|u|^p\, dS,
\]
for some $C$ depending only on $p$ and $A$, and hence that
$\|F_p\|_{L^\infty((r_b,R_b])}<\infty$. Theorem~\ref{thm:uniquecont}
implies that $u$ satisfies the unique continuation principle.

Theorem~\ref{thm:weakDprop} tells that the boundedness of the
frequency function implies \eqref{eq:weakD1} and, more importantly,
the weak doubling property \eqref{eq:weakD2}. In the following, we
shall show that also the converse is true in a situation in which a
certain additional assumption, which is valid in the case $p=2$, is
satisfied. Suppose inequality \eqref{eq:weakD1} holds for every
$r_1,\,r_2\in (r_b,r^\star]$. Assume further that there exists a
positive constant $A<\infty$ such that
\begin{equation} \label{eq:condII}
\int_{B_r}|u|^p\, dx \leq Ar\int_{\partial B_r}|u|^p\, dS,
\end{equation}
where $B_r \subset G$. Let $u$ be a solution to the \p-Laplace
equation. It therefore satisfies a Caccioppoli type estimate (e.g. \cite[Lemma 2.9]{Lindqvist}). More
precisely, there exists a positive constant $C<\infty$, depending on
$p$, such that for all $B_r\subset B_\rho\subset G$ we have
\begin{equation*} \label{eq:Caccioppoli}
\int_{B_r}|\nabla u|^p\, dx \leq \frac{C}{(\rho-r)^p}\int_{B_\rho}|u|^p\, dx.
\end{equation*}
Let $r\in (r_b,r^\star]$. The weak doubling property
\eqref{eq:weakD2}, the Caccioppoli estimate, and \eqref{eq:condII}
altogether imply the following estimate
\begin{equation*}
  F_p(r) = \frac{r\int_{B_r}|\nabla u|^p\, dx}{\int_{\partial
      B_r}|u|^p\, dS}
  \leq \frac{Cr}{(r^\star-r)^p}\frac{\int_{B_{r^\star}}|u|^p\, dx}{\int_{\partial B_{r^\star}}|u|^p\, dS} \leq C\frac{rr^\star}{(r^\star-r)^p},
\end{equation*}
where the constant $C$ depends on $n$, $p$, and $A$. To conclude, the
frequency function remains bounded as $r$ tends to $r_b$.

We close the paper by remarking that convexity of $\int_{B_r}|u|^p\,
dx$ implies \eqref{eq:condII} with $A=1$. For harmonic functions,
moreover, it is easy to prove that both
\[
\int_{B_r}u^2\, dx  \quad \textrm{and}\quad \int_{\partial B_r}u^2\,
dS
\]
are indeed convex in $\R^n$, $n\geq 2$. Convexity of the latter
follows by showing that
\[
\frac{d}{dr}\left(\frac1{r}\int_{\partial B_r}u^2\, dS\right) \geq 0.
\]

\end{document}